\documentclass[a4paper,12pt]{article}
\usepackage{amsmath,amsthm,amssymb,amsfonts}
\usepackage{bbm,bm}
\usepackage{latexsym}
\usepackage{mathrsfs}
\usepackage{threeparttable}
\usepackage{tabularx}
\usepackage{booktabs}
\usepackage{graphicx,subfigure}
\usepackage{color}
\usepackage{indentfirst}
\usepackage{geometry}
\usepackage{caption}
\usepackage{lineno}
\usepackage{abstract}
\usepackage{enumerate,mdwlist}
\usepackage[numbers,sort&compress]{natbib}
\usepackage[colorlinks,
            linkcolor=blue, 
            citecolor=red  
            ]{hyperref}
\usepackage{epstopdf}

\def \[{\begin{equation}}
\def \]{\end{equation}}

\newtheorem{thm}{Theorem}[section]

\newtheorem{lem}[thm]{Lemma}
\newtheorem{cor}[thm]{Corollary}

\newtheorem{fact}[thm]{Fact}
\newtheorem{obs}[thm]{Observation}

\newtheorem{rem}[thm]{Remark}

\begin{document}

\setlength{\baselineskip}{20pt}
\begin{center}{\Large \bf Maximizing the Minimum and Maximum Forcing Numbers of Perfect Matchings of Graphs}\footnote{This work was supported by National Natural Science Foundation of China (Grant No. 11871256) and Gansu Provincial Department of Education: Youth Doctoral fund project (Grant No. 2021QB-090)}
\vspace{4mm}

{Qian Qian \uppercase{liu} \quad \quad \quad He Ping \uppercase{Zhang}\footnote{The corresponding author.}}
\vspace{2mm}

\author{Qian Qian \uppercase{liu} \quad \quad \quad He Ping \uppercase{Zhang}\thanks{*Corresponding author}\footnotemark}
    {School of Mathematics and Statistics, Lanzhou University, Lanzhou 730000, P. R. China\\
    E-mail\,$:$ liuqq2016@lzu.edu.cn  \quad \quad \quad  zhanghp@lzu.edu.cn}
\vspace{2mm}
\end{center}
\noindent {\bf Abstract}: Let $G$ be a simple graph with $2n$ vertices and a perfect matching. The forcing number $f(G,M)$ of a perfect matching $M$ of $G$ is the smallest
cardinality of a subset of $M$ that is contained in no other perfect matching of $G$. Among all perfect matchings $M$ of $G$, the minimum and maximum values of
$f(G,M)$ are called the minimum and maximum forcing numbers of $G$, denoted by $f(G)$ and $F(G)$, respectively. Then $f(G)\leq F(G)\leq n-1$. Che and Chen (2011)
proposed an open problem: how to characterize the graphs $G$ with $f(G)=n-1$. Later they showed that for a bipartite graph $G$, $f(G)=n-1$ if and only if $G$ is a
complete bipartite graph $K_{n,n}$. In this paper, we completely solve the problem of Che and Chen, and show that $f(G)=n-1$ if and only if $G$ is a complete multipartite graph or a graph obtained from complete bipartite graph $K_{n,n}$ by adding  arbitrary  edges in the same partite set. For all graphs $G$ with $F(G)=n-1$, we prove that the forcing spectrum of each such graph $G$ forms an integer interval by matching 2-switches and the minimum forcing numbers of all such graphs $G$ form an integer interval from $\lfloor\frac{n}{2}\rfloor$ to $n-1$.

\vspace{2mm} \noindent{\it Keywords}: Perfect matching, minimum forcing number, maximum forcing number, forcing spectrum, complete multipartite graph
\vspace{2mm}


 {\setcounter{section}{0}
\section{\normalsize Introduction}\setcounter{equation}{0}
We only consider finite and simple graphs.
 Let $G$ be a graph with vertex set $V(G)$ and edge set $E(G)$. The \emph{order} of $G$ is the number of vertices in $G$. A graph is \emph{trivial} if it contains
 only one vertex. Otherwise, it is \emph{non-trivial}. The \emph{degree} of vertex $v$ in $G$, written $d_G(v)$, is the number of edges incident to $v$. An
 \emph{isolated vertex} is a vertex of degree 0.
 The maximum degree and the minimum degree of $G$ are denoted by $\Delta(G)$ and $\delta(G)$, respectively. If all vertices of $G$ have degree $k$, then $G$ is
 \emph{$k$}-\emph{regular}.
 A complete graph of order $n$ is denoted by $K_{n}$. Let $P_n$ be a path with $n$  vertices and $\overline{P_n}$ be the complement of $P_n$.

For an edge subset $F$ of $G$, we write $G-F$ for the subgraph of $G$ obtained by deleting the edges in $F$. If $F=\{e\}$, we simply write $G-e$ instead of
$G-\{e\}$. For a vertex subset $T$ of $G$, we write $G-T$ for the subgraph of $G$ obtained by deleting all vertices in $T$ and their incident edges. If $T=\{v\}$ is
a singleton, we write $G-v$ rather than $G-\{v\}$.
For a vertex subset $T$ of $G$, we write $G[T]$ for the subgraph $G-(V(G)\setminus T)$, induced by $T$. For a graph $H$, $G$ is \emph{$H$-free} if it contains no
$H$ as an induced subgraph.

\emph{A perfect matching} of a graph is a set of disjoint edges covering all vertices of the graph. A graph $G$ is \emph{factor-critical} if $G-u$ has a perfect
matching for every vertex $u$ of $G$.
A graph $G$ is \emph{bicritical} if $G$ contains an edge and $G-u-v$ has a perfect matching for every pair of distinct vertices $u$ and $v$ in $G$. A 3-connected
bicritical graph is called a \emph{brick}.
For a nonnegative integer $l$, a connected graph $G$ with at least $2l+2$ vertices is \emph{$l$-extendable} if $G$ has a perfect matching and every matching of size
$l$ is contained in a perfect matching of $G$.

A perfect matching coincides with a Kekul\'{e} structure in organic chemistry or a dimer covering in statistic physics. Klein and Randi\'{c} \cite{3} proposed the
``\emph{innate degree of freedom}'' of a Kekul\'{e} structure, i.e., the least number of double bonds can determine this entire Kekul\'{e} structure, which plays an
important role in resonant theory. Afterwards, it was called the forcing number by Harary et al. \cite{2}.
 A \emph{forcing set} for a perfect matching $M$ of $G$ is a subset of $M$ that is contained in no other perfect matching of $G$. The smallest cardinality of a
 forcing set of $M$ is called the \emph{forcing number} of $M$, denoted by $f(G,M)$.

Let $G$ be a graph with a perfect matching $M$. A cycle of $G$ is \emph{M-alternating} if its edges appear alternately in $M$ and $E(G)\setminus M$. If $C$ is an $M$-alternating cycle of $G$, then the symmetric difference $M\oplus E(C):=(M\setminus E(C))\cup(E(C)\setminus M)$ is another perfect matching of $G$.
We use $V(S)$ to denote the set of all end vertices in an edge subset $S$ of $E(G)$.
An equivalent condition for a forcing set of a perfect matching was mentioned by Pachter and Kim as follows.
\begin{lem}[\cite{6}] \label{2.1}Let $G$ be a graph with a perfect matching $M$. Then a subset $S\subseteq M$ is a forcing set of $M$ if and only if $G-V(S)$ contains
no $M$-alternating cycles.
\end{lem}
Let $c(M)$ denote the maximum number of disjoint $M$-alternating cycles in $G$. Then $f(G,M)\geq c(M)$ by Lemma \ref{2.1}.
For plane bipartite graphs, Pachter and Kim obtained the following minimax theorem.
\begin{thm}[\cite{6}]\label{thm6} Let $G$ be a plane bipartite graph. Then $f(G, M)=c(M)$ for any perfect matching $M$ of $G$.
\end{thm}

The \emph{minimum} (resp. \emph{maximum}) \emph{forcing number} of $G$ is the minimum (resp. maximum) value of $f(G,M)$ over all perfect matchings $M$ of $G$, denoted
by $f(G)$
(resp. $F(G)$). Adams et al. \cite{4} showed that determining a smallest forcing set of a perfect matching is NP-complete for bipartite graphs with the maximum degree 3. Using this result,
Afshani et al. \cite{5} proved that determining the minimum forcing number is NP-complete for bipartite graphs with the maximum degree 4. However, the computational complexity
of the maximum forcing number of a graph is still an open problem \cite{5}.

Xu et al. \cite{27} showed that the maximum forcing number of a hexagonal system is equal to its resonant number.  The same result also holds  for a polyomino graph
\cite{64,LW17} and for a BN-fullerene graph  \cite{40}.
In general, for 2-connected plane bipartite graphs, the resonant number can be computed in polynomial time (see Ref. \cite{13} due to Abeledo and Atkinson).
Hence, the maximum forcing numbers of such three classes of graphs can be computed in polynomial time.

Moreover, some minimax results have been obtained \cite{63, 64}: for each  perfect matching $M$ of a hexagonal system $G$ with $f(G,M)=F(G)$,  there exist $F(G)$
disjoint $M$-alternating hexagons in $G$; for every perfect matching $M$ of polyomino graphs $G$ with $f(G,M)=F(G)$ or $F(G)-1$,  $f(G,M)$ is equal to the maximum
number of disjoint $M$-alternating squares in $G$.

Zhang and Li \cite{58}, and Hansen and Zheng \cite{HZ94} independently determined the  hexagonal systems $G$ with $f(G)=1$, and  Zhang and Zhang \cite{61} gave a
generalization to plane bipartite graphs $G$ with $f(G)=1$. For 3-connected cubic graphs $G$ with $f(G)=1$, Wu et al. \cite{60} showed that it can be generated
from $K_4$ via $Y \rightarrow  \Delta$-operation.
For a convex hexagonal system $H(a_1,a_2,a_3)$ with a perfect matching, recently Zhang and Zhang \cite{57} proved that its minimum forcing number is equal to
min$\{a_1, a_2, a_3\}$ by monotone path systems.

For $n$-dimensional hypercube $Q_n$, Pachter and Kim \cite{6} conjectured that $f(Q_n)=2^{n-2}$ for integer $n\geq 2$. Later Riddle \cite{7} confirmed it for even $n$
by the trailing vertex method. Recently, Diwan \cite{25} proved that the conjecture holds by linear algebra.
Using well-known Van der Waerden theorem, Alon  showed that  $F(Q_n)>c2^{n-1}$ for any constant $0<c<1$ and sufficient large $n$ (see \cite{7}).
There are also some researches about the minimum or maximum forcing numbers of other graphs, such as grids \cite{5,6,16,29,9}, fullerene graphs
\cite{19,17,40,62}, toroidal polyhexes \cite{8}, toroidal and Klein bottle lattices \cite{29,7,16}, etc.

We denote by $\mathcal{G}_{2n}$ the set of all graphs of order $2n$ and with a perfect matching. Let $G\in \mathcal{G}_{2n}$. Then each perfect matching of $G$ has $n$ edges and any $n-1$ edges among it form a forcing set. So we have that $f(G)\leq F(G)\leq n-1$.
Che and Chen \cite{1} proposed how to characterize the graphs $G$ with $f(G)=n-1$. Afterwards, they \cite{10} solved the problem for bipartite graphs and obtained the
following result.
\begin{thm}[\cite{10}]\label{thm1.3} Let $G$ be a bipartite graph with $2n$ vertices. Then $f(G)=n-1$ if and only if $G$ is complete bipartite graph $K_{n,n}$.
\end{thm}

In this paper, we solve the problem of Che and Chen. Let $\mathcal{K}_{n,n}^+$ be a family of graphs obtained by adding arbitrary additional edges in the same partite set to the complete bipartite graph $K_{n,n}$. In Section 2, we discuss some basic properties for graphs $G\in\mathcal{G}_{2n}$ with $F(G)=n-1$, and obtain that $G$ is $n$-connected, 1-extendable except for graphs in $\mathcal{K}_{n,n}^+$. In particular, we give a characterization for a perfect matching $M$ of $G$ with $f(G,M)=n-1$. In Section 3, we answer the problem proposed by Che and Chen, and obtain that $f(G)=n-1$ if and only if $G$ is either a complete multipartite graph or a graph in $\mathcal{K}_{n,n}^+$. In Section 4, for all such 1-extendable graphs $G$ with $F(G)=n-1$, we determine which of them are not 2-extendable.
Finally in Section 5 we show that $f(G)\geq \lfloor\frac{n}{2}\rfloor$ for any graph $G\in\mathcal{G}_{2n}$ with $F(G)=n-1$, and the minimum forcing numbers of all such graphs form an
integer interval $[\lfloor\frac{n}{2}\rfloor,n-1]$. Also we prove that the forcing spectrum (set of forcing numbers of all perfect matchings) of each such graph $G$ forms an integer interval.
\section{\normalsize  Some basic properties of graphs $G\in \mathcal{G}_{2n}$ with $F(G)=n-1$}
Let $G\in \mathcal{G}_{2n}$ with $F(G)=n-1$. In this section, we will obtain some basic properties of graph $G$.
By definition of forcing numbers, we obtain the following observation.
\begin{obs}\label{obs}Let $G\in \mathcal{G}_{2n}$. Then

\noindent(\romannumeral1) $F(G)=n-1$ if and only if $f(G, M)=n-1$ for some perfect matching $M $ of $G$, and

\noindent(\romannumeral2) $f(G)=n-1$ if and only if $f(G, M)=n-1$ for every perfect matching $M$ of $G$.
\end{obs}

We call a graph $G\in \mathcal{G}_{2n}$ with $F(G)=n-1$ \emph{minimal} if $F(G-e)\leq n-2$ for each edge $e$ of $G$.
Next we give a characterization for a perfect matching $M$ of $G$ with $f(G,M)=n-1$.
\begin{lem}\label{2.2}Let $G\in \mathcal{G}_{2n}$ for $n\geq 2$. Then $G$ has a perfect matching $M$ such that $f(G,M)=n-1$ if and only if $G[V(\{e_i,e_j\})]$
contains an
$M$-alternating cycle for any two distinct edges $e_i$ and $e_j$ of $M$. Moreover, $G$ is minimal if and only if $G[V(\{e_i,e_j\})]$ is exactly an $M$-alternating
4-cycle for any two distinct edges $e_i$ and $e_j$ of $M$.
\end{lem}
\begin{proof}(1) Suppose that $f(G,M)=n-1$. Then $M\setminus\{{e_i,e_j}\}$ is not a forcing set of $M$ for any two distinct
edges $e_i$ and $e_j$ of $M$. By Lemma \ref{2.1}, $G[V(\{e_i,e_j\})]$ contains an $M$-alternating cycle.
 Conversely, for any subset $S$ of $M$ with size less than $n-1$, there are two distinct edges in $M\setminus S$. By the assumption, $G-V(S)$ contains an $M$-alternating cycle. By Lemma \ref{2.1}, $S$ is not a forcing set of $M$, that
 is, $f(G,M)\geq n-1$.

(2) Suppose that $G$ is minimal. Since $f(G,M)=n-1$, $G[V(\{e_i,e_j\})]$ contains an $M$-alternating
4-cycle for any two distinct edges $e_i$ and $e_j$ of $M$. Moreover, $G[V(\{e_i,e_j\})]$ is exactly an $M$-alternating 4-cycle. If not, then $G[V(\{e_i,e_j\})]$ is
isomorphic to $K_4$ or a graph obtained from $K_4$ by deleting an edge, say $u_iu_j$, where $e_l=u_lv_l$ for $l=i,j$. Then $G[V(\{e_i,e_j\})]-{v_iv_j}$ contains an
$M$-alternating cycle. So $F(G-v_iv_j)=f(G-v_iv_j,M)=n-1$, which contradicts the minimality of $G$.

Conversely, by the assumption, we have $f(G,M)=n-1$ where $M=\{e_l=u_lv_l|l=1,2,\dots,n\}$ is a perfect matching of $G$. By Observation \ref{obs}, $F(G)=n-1$.
Next we are to prove that $G$ is minimal. Let $G'=G-e$ where $e$ is an arbitrary edge of $G$. We can show that $G'$ has a perfect matching. If $e\notin M$, then $M$ is a perfect matching of $G'$. If $e=e_i\in M$ for some $1\leq i\leq n$, then $G[\{u_i,v_i,u_j,v_j\}]$ is exactly an $M$-alternating 4-cycle $C$ for any $j\neq i$ and $1\leq j\leq n$. Then $M\oplus E(C)$ is a perfect matching of $G'$. For any perfect matching $M'$ of $G'$, we will show that $f(G', M')\leq n-2$,
and so $G$ is minimal.

Without loss of generality, any edge $e$ of $G$ can be represented as either $u_iv_i$ or $u_iv_j$ where $j\neq i$ because any edge $u_jv_j$ can be written as $v_ju_j$
by switching two end vertices. Let $w\in\{u,v\}$.

{\textbf{Case 1.} $e=u_iv_i$.}

Then $\{u_iw_j,v_iw_k\}\subseteq M'$ for some integers $j, k$ different from $i$. It follows that $v_iw_j\notin E(G')\subset E(G)$ since $G[\{u_i,v_i,u_j,v_j\}]$ is a
4-cycle by the assumption. Then $G[\{u_i,w_j,v_i,w_k\}]$ does not contain a 4-cycle of $G'$ since $v_i$ cannot be adjacent to either $u_i$ or $w_j$ in $G'$. By Lemma
\ref{2.1}, $M'\setminus\{u_iw_j,v_iw_k\}$ is a forcing set of $M'$. So, $f(G',M')\leq n-2$.

{\textbf{Case 2.} $e=u_iv_j$ where $j\neq i$.}

\textbf{Subcase 2.1.} Both $u_iv_i$ and $u_jv_j$ are contained in $M'$.  By the assumption, $G[\{u_i,v_i,u_j,v_j\}]$ is an $M$-alternating 4-cycle of $G$ where
$\{u_iv_i,u_jv_j\}\subseteq M$. So $G'[\{u_i,v_i,u_j,v_j\}]$ is just a path of length three since $G'=G-e=G-u_iv_j$. By
Lemma \ref{2.1}, $M'\setminus \{u_iv_i, u_jv_j\}$ is a forcing set of $M'$. So, $f(G', M')\leq n-2$.

\textbf{Subcase 2.2.} At least one of $u_iv_i$ and $u_jv_j$ is not contained in $M'$. Without loss of generality, assume that $u_jv_j\notin M'.$
Since $e=u_iv_j$ is not an edge of $G'$, there exists a vertex $w_l\neq u_i$ such that $v_jw_l\in M'$. Note that $v_iv_j$ cannot be an edge of $G'\subset G$ since
$G[\{u_i,v_i,u_j,v_j\}]$ is an $M$-alternating 4-cycle of $G$. Then $w_l\neq v_i$. If $v_iw_l\notin E(G')$, then $v_i$ is adjacent to neither $w_l$ nor $v_j$ in $G'$. So
$G'[\{v_i,w_t,w_l,v_j\}]$ contains no $M'$-alternating cycles where $v_iw_t\in M'$. By Lemma \ref{2.1}, $M'\setminus \{v_iw_t,v_jw_l\}$ is a forcing set of $M'$.
Thus, $f(G',M')\leq n-2$. If $v_iw_l\in E(G')$, then $u_iw_l\notin E(G')\subset E(G)$ since $G[\{v_i,u_i,u_l,v_l\}]$ is exactly an $M$-alternating 4-cycle by the assumption.
Since $G'=G-e$, we have $e=u_iv_j\notin E(G')$. So $u_i$ is adjacent to neither $w_l$ nor $v_j$ in $G'$, and
$G'[\{u_i,w_k,w_l,v_j\}]$ contains no $M'$-alternating cycles where $u_iw_k\in M'$. By Lemma \ref{2.1}, $M'\setminus \{u_iw_k,v_jw_l\}$ is a forcing set of $M'$. So, $f(G',M')\leq n-2$.
\end{proof}

By Lemma \ref{2.2}, adding some extra edges to a minimal graph $G$ maintains the same maximum forcing number as $G$.
Clearly, $K_{n,n}$ is minimal for $n\geq 1$ by Lemma \ref{2.2}. For $n\geq 3$, we have other such minimal graphs shown in Fig. \ref{p82}.
\begin{figure}
\centering
\includegraphics[height=2.6cm,width=14cm]{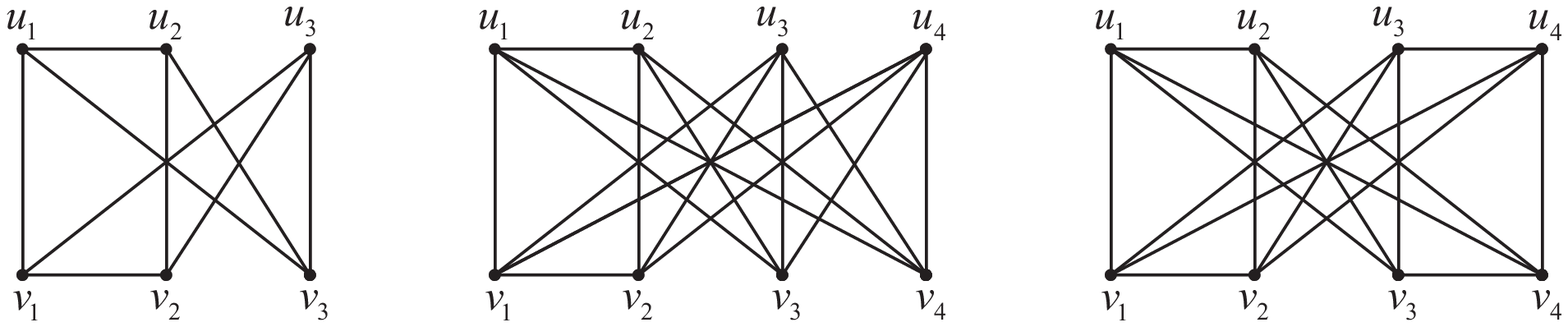}
\caption{\label{p82}Some examples of minimal graphs with $n=3$ and 4.}
\end{figure}

Let $G$ be a graph with a perfect matching. An edge $e$ of $G$ is called a \emph{fixed double bond}
if $e$ is contained in all perfect matchings of $G$. The connectivity and edge connectivity of $G$ are denoted by $\kappa(G)$ and $\lambda(G)$, respectively.

\begin{lem}\label{p1} Assume that $G\in \mathcal{G}_{2n}$ has $F(G)=n-1> 0$. Then

\noindent(\romannumeral1) $G$ has no fixed double bond, that is, $G-e$ has a perfect matching for each edge $e$ of $G$.

\noindent(\romannumeral2) $\kappa(G)\geq n$. Moreover, if $G$ is minimal, then $G$ is $n$-regular, and $\kappa(G)=\lambda(G)=n$.
\end{lem}
\begin{proof}(\romannumeral1) It has been implied in the sufficiency part of (2) in the proof of Lemma \ref{2.2}.

(\romannumeral2) Since $F(G)=n-1$, there exists a perfect matching $M$ of $G$ such that $f(G,M)=n-1$ where $M=\{u_iv_i|i=1,2,\dots,n\}$. For any $X\subseteq V(G)$ with $|X|<n$, there is at least one
pair of vertices $u_i$ and $v_i$ not in $X$ for some $1\leq i\leq n$. Then for the other vertices not in $X$, say $u_j$ (resp. $v_j$), either $u_iu_j$ or $v_iu_j$
(resp. $u_iv_j$ or $v_iv_j$) is contained in $E(G)$ by Lemma \ref{2.2}. Hence $G-X$ is connected and so $G$ is $n$-connected. Therefore, $\kappa (G)\geq n$.

 If $G$ is minimal, then for any $1\leq i\leq n$, $u_i$ (resp. $v_i$) is adjacent to exactly one of $u_j$ and $v_j$ for any $1\leq j\leq n$ and $ j\neq i$ by Lemma \ref{2.2}.
 Combining that $u_iv_i$ is an edge, we obtain that $u_i$ and $v_i$ are of degrees $n$. So $G$ is $n$-regular.  Combining that $n\leq \kappa(G)\leq\lambda(G)\leq\delta(G)=n$, we obtain
 that $\kappa(G)=\lambda(G)=n$.
 \end{proof}

The number of odd components of $G$ is denoted by $o(G)$. The following result gives an equivalent condition for a bicritical graph.
\begin{lem}[\cite{14}]\label{l3} A graph $G$ is bicritical if and only if for any $X \subseteq V (G)$ and $|X|\geq 2$, $o(G-X)\leq |X|-2$.
\end{lem}

Next we will show that $\mathcal{K}^+_{n,n}$ is a subclass of graphs $G\in \mathcal{G}_{2n}$ with $F(G)=n-1$ such that $G$ contains an independent set of size $n$.
\begin{lem}\label{p2}Assume that $G\in \mathcal{G}_{2n}$ has $F(G)=n-1$.
Then $G$ is a graph in $\mathcal{K}^+_{n,n}$ if and only if $G$ contains an independent set of size $n$. Otherwise, $G$ is a brick, and thus 1-extendable.
 \end{lem}
\begin{proof}(1) The necessity is obvious. We prove sufficiency next. Since $F(G)=n-1$, there exists a perfect matching $M$ of $G$ such that $f(G,M)=n-1$. Let $\{u_1,u_2,\dots,u_n\}$ be an independent set of size $n$ in $G$ and let $M=\{u_iv_i|i=1,2,\dots,n\}$. Then $u_iu_j\notin E(G)$ for any $1\leq i<j\leq n$. By Lemma \ref{2.2}, $\{u_iv_j,v_iu_j\}$ is contained in $E(G)$. Hence, $G$ is some graph in $\mathcal{K}_{n,n}^+$, for there may be some other edges with both end vertices in $\{v_1,v_2,\dots,v_n\}$.

(2) If $G$ is not a graph in $\mathcal{K}^+_{n,n}$, then we will prove that $G$ is a brick. For $n\geq 3$, $G$ is 3-connected by Lemma \ref{p1}. For $n\leq 2$, exactly one graph $K_4$ that is not in $\mathcal{K}^+_{n,n}$ is 3-connected. Next we prove that $G$ is bicritical. Suppose that $X\subseteq V(G)$ with $|X|\geq 2$. If $|X|\leq n-1,$ then
$G-X$ is connected by
Lemma \ref{p1}. Hence $o(G-X)\leq 1\leq |X|-1$. Otherwise, we have $|X|\geq n$. Then $o(G-X)\leq |V(G-X)|\leq n$. Since $G$ is not a graph in $\mathcal{K}_{n,n}^+$, $G$ contains no
independent set of size $n$. Hence $o(G-X)\leq n-1\leq|X|-1$. Since $G$ is of even order, $o(G-X)$ and $|X|$ are of the same parity. So $o(G-X)\leq |X|-2$. By Lemma
\ref{l3}, $G$ is bicritical.
\end{proof}
\section{\normalsize Graphs $G\in \mathcal{G}_{2n}$ with the minimum forcing number $n-1$}
In this section, we will determine all graphs $G\in \mathcal{G}_{2n}$ with $f(G)=n-1$ to completely solve the problem proposed by Che and Chen \cite{1}.

 Tutte's theorem states that $G$ has a perfect matching if and only if $o(G-S)\leq |S|$ for any $S\subseteq
V(G)$. By Tutte's theorem, Yu \cite{34} obtained an equivalent condition for a connected graph with a perfect matching that is not $l$-extendable.
Recently, for an $(l-1)$-extendable graph $G$ with $l\geq 1$, Alajbegovi\'{c} et al. \cite {28} obtained that $G$ is not $l$-extendable if and only if there exists a
subset $S\subseteq V(G)$ such that $G[S]$ contains $l$ independent edges and $o(G-S)=|S|-2l+2$. In fact, $S$ can be chosen so that each component of $G-S$ is
factor-critical (see Theorem 2.2.3 in \cite{20}). Combining these, we obtain the following result.
\begin{lem}\label{l4}Let $l\geq1$ be an integer and $G$ be an $(l-1)$-extendable graph of order at least $2l+2$. Then $G$ is not $l$-extendable if and only if there
exists a subset $S\subseteq V(G)$ such that $G[S]$ contains $l$ independent edges, all components of $G-S$ are factor-critical, and $o(G-S)=|S|-2l+2$.
\end{lem}

A \emph{complete multipartite graph} is a graph whose vertices can be partitioned into sets so that $u$ and $v$ are adjacent if and only if $u$ and $v$ belong to
different sets of the partition.
We write $K_{n_1,n_2,\dots,n_k}$ for the complete $k$-partite graph with partite sets of sizes $n_1,n_2,\dots,n_k$. In fact, a complete multipartite graph is a
$\overline{P_3}$-free graph (see Exercise 5.2.2 in \cite{36}).
\begin{lem}[\cite{36}]\label{complete} A graph is $\overline{P_3}$-free if and only if it is a complete multipartite graph.
\end{lem}
  \begin{thm}\label{thm1}Let $G\in \mathcal{G}_{2n}$. Then  $f(G)=n-1$ if and only if $G$ is a complete multipartite graph with each partite set having size no
  more than $n$ or $G$ is some graph in $\mathcal{K}_{n,n}^+$.
 \end{thm}
 \begin{proof}Sufficiency.
 Suppose that $G$ is some graph in $\mathcal{K}^+_{n,n}$. Then any perfect matching $M$ of $G$ is also a perfect matching of $K_{n,n}$. By Lemma \ref{2.2}, $f(G,M)=n-1$. By the arbitrariness
 of $M$, $f(G)=n-1$.

 Suppose that $G=K_{n_1,n_2,\dots,n_k}$ is a complete multipartite graph where $1\leq n_i\leq n$ for $1\leq i\leq k$ and $k\geq 2$. First we will show that $G$ has a
 perfect matching. For any nonempty subset $S$ of $V(G)$, if $G-S$ has at least two vertices from different partite sets of $G$, then $G-S$ is connected and
 $o(G-S)\leq 1\leq |S|$. Otherwise, all vertices of $G-S$ belong to one partite set of $G$. Then $|V(G-S)|\leq n$ and $|S|\geq n$. Hence $o(G-S)\leq n\leq |S|$. For
 $S=\emptyset$, $o(G-S)=o(G)=|S|$. By Tutte's theorem, $G$ has a perfect matching.

 Clearly, the result holds for $n=1$. Next let $n\geq 2$. Suppose to the contrary that $f(G)\leq n-2$. Then there exists a perfect matching $M$ of $G$ and a minimum
 forcing set $S$ of $M$ such that $|S|=f(G,M)=f(G)$. By Lemma \ref{2.1}, $G-V(S)$ contains no $M$-alternating cycles.
 So there are two distinct edges $\{u_iv_i,u_jv_j\}\subseteq M\setminus S$ and $G[\{u_i,v_i,u_j,v_j\}]$ contains no $M$-alternating cycles. That is to say, neither
 $\{u_iu_j,v_iv_j\}$ nor $\{u_iv_j,v_iu_j\}$ is contained in $E(G)$. Without loss of generality, we assume that none of $u_iu_j$ and $u_iv_j$ belong to $E(G)$.  Then
 $G[\{u_i,u_j,v_j\}]$ is isomorphic to $\overline{P_3}$, which contradicts Lemma \ref{complete}.

 Necessity. Suppose that $f(G)=n-1$. If $G$ is a complete multipartite graph, then each partite set of $G$ has size no more than $n$ for $G$ has a perfect matching.
 If $G$ is not a complete multipartite graph, then by Lemma \ref{complete}, $G$ contains an induced subgraph $H$ isomorphic to $\overline{P_3}$.
 Note that the edge $e$ of $H$ is not in any perfect matching of $G$. Otherwise, there is a perfect matching $M$ of $G$ containing $e$. By Observation \ref{obs}, $f(G,M)=n-1$. Let $v$ be the
 vertex of $H$ except for both end vertices of $e$, and $e'$ be the edge of $M$ incident with $v$. Then $G[V(\{e,e'\})]$ contains no $M$-alternating cycles for $v$ is
 not incident with any end vertices of $e$, which contradicts Lemma \ref{2.2}.
 So $G$ is not 1-extendable. By Lemma \ref{l4}, there exists $S\subseteq V(G)$ such that $G[S]$ contains an edge, all components of $G-S$ are factor-critical, and
 $o(G-S)=|S|\geq 2$.

 We claim that all components of $G-S$ are singletons. Otherwise, assume that $C_1$ is a non-trivial component of $G-S$. Let $M$ be a perfect matching of $G$. By Observation \ref{obs}, $f(G,M)=n-1$. Since $o(G-S)=|S|$, $M$ matches $S$ to distinct components of $G-S$. Assume that $e_1$ is an edge in  $M\cap E(C_1)$ and $e_2$ is an edge of $M$ which connects a vertex of $S$ and a vertex of another
 component $C_2$ of $G-S$. Then $G[V(\{e_1,e_2\})]$ contains no $M$-alternating cycles, which contradicts Lemma \ref{2.2}. Therefore, each component of $G-S$ is a
 singleton and so $|S|+o(G-S)=2n$. It follows that $o(G-S)=n$ since we have shown that $o(G-S)=|S|$. So $G$ contains an independent set of size $n$. By Lemma
 \ref{p2}, $G$ is a graph in $\mathcal{K}^+_{n,n}$.
\end{proof}

Taking $n=3$ for example, $K_{3,3},K_{3,2,1},K_{3,1,1,1},K_{2,2,2},K_{2,2,1,1},K_{2,1,1,1,1},K_6$ are all complete multipartite graphs with each partite set having
 size no more than 3 and $K_{3,3}+e$ is the unique graph in $\mathcal{K}^+_{3,3}$ except for above complete multipartite graphs, where $e$ is an edge connecting any two
 nonadjacent vertices of $K_{3,3}$.
\section{\normalsize Extendability of graphs $G\in\mathcal{G}_{2n}$ with $F(G)=n-1$}
Let $G$ be a graph in $\mathcal{G}_{2n}$ with $F(G)=n-1$ and be different from graphs in $\mathcal{K}_{n,n}^+$. By Lemma \ref{p2}, $G$ is 1-extendable. However, it is not necessarily 2-extendable.
 We know that an $l$-extendable graph is $(l-1)$-extendable for an integer $l\geq 1$, and 2-extendable graphs are either bricks or braces (2-extendable bipartite
 graphs) \cite{43}. In this section, we will determine which graphs in Lemma \ref{p2} are 1-extendable but not 2-extendable.
\begin{thm}\label{p3}Let $G\in\mathcal{G}_{2n}$ with $F(G)=n-1$ where $n\geq 3$. Then
$G$ is 1-extendable but not 2-extendable if and only if $G$ has a
perfect matching $M=\{u_iv_i| i=1,2,\dots,n\}$ with $f(G,M)=n-1$ so that one of the following conditions holds.

(\romannumeral1) $n\geq 4$, $G[\{v_1,v_2,\dots,v_n\}]$ consists of one triangle and $n-3$ isolated vertices, and  $G[\{u_1,u_2,\dots,u_n\}]$ has at least two
independent edges (see an example in Fig. \ref{p91111}(a)).

(\romannumeral2) $\{v_1,v_2,\dots,v_{n-1}\}$ is an independent set, $G[\{u_1,u_2,\dots,u_n,v_n\}]$ has at least two independent edges, and
$\{v_iv_n,v_ju_n\}\subseteq E(G)$ for some $i$ and $j$ with $1\leq i,j\leq n-1$ (see examples (b) and (c) in Fig. \ref{p91111}).
\begin{figure}
\centering
\includegraphics[height=3.2cm,width=14cm]{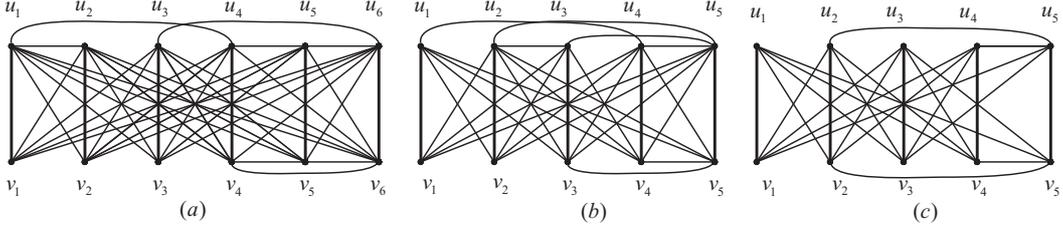}
\caption{\label{p91111}Three examples for non-2-extendable graphs.}
\end{figure}
\end{thm}
\begin{proof}Sufficiency. Since $f(G,M)=n-1$, by Lemma \ref{p2}, $G$ is 1-extendable or $G$ contains an independent set of size $n$. We claim that $G$ contains no independent set of size $n$. If we have done, then $G$ is 1-extendable. Let $S=\{u_1,u_2,\dots,u_n\}$ (resp. $\{u_1,u_2,\dots,u_n,v_n\}$) be a subset of $V(G)$
corresponding to (\romannumeral1) (resp. (\romannumeral2)). Then $G[S]$ contains two independent edges, all components of $G-S$ are factor-critical, and
$o(G-S)=|S|-2$. By Lemma \ref{l4}, $G$ is not 2-extendable.

Now we prove the claim. Let $I$ be any independent set of $G$. We will prove that $|I|\leq n-1$. We consider the graphs $G$ satisfying (\romannumeral1). Let $\{v_1,v_2,\dots,v_{n-3}\}$ be the set of $n-3$ isolated vertices and $G[\{v_{n-2},v_{n-1},v_n\}]$ be the triangle of $G[\{v_{1},v_{2},\dots,v_n\}]$. If $I\subseteq\{u_1,u_2,\dots,u_n\}$, then $|I|\leq n-2$ for $G[\{u_1,u_2,\dots,u_n\}]$ has at
least two independent edges. Otherwise, there exists $v_i\in I$ for some $1\leq i\leq n$.
For any $1\leq j\leq n$ and $j\neq i$, $G[\{u_i,v_i,u_j,v_j\}]$ contains an $M$-alternating 4-cycle by Lemma \ref{2.2}. So if $v_iv_j\notin E(G)$ for some $j$ then $v_iu_j\in E(G)$ and $u_j\notin I$. Hence, if $1\leq i\leq n-3$ then $u_j\notin I$ for any $1\leq j\leq n$ and $j\neq i$ since $v_iv_j\notin E(G)$.
Hence $I\subseteq \{v_1,v_2,\dots,v_n\}$ and $|I|\leq n-2$. If $n-2\leq i\leq n$, then $u_j\notin I$ for any $1\leq j\leq n-3$ since $v_iv_j\notin E(G)$. Hence $I\subseteq \{v_1,v_2,\dots,v_{n-3},v_i\}\cup(\{u_{n-2},u_{n-1},u_n\}\setminus \{u_i\})$.
If $v_1\in I$, then $\{u_{n-2},u_{n-1},u_n\}\cap I=\emptyset$ by Lemma \ref{2.2}, so $|I|\leq n-2$. Otherwise, $v_1\notin I$. Then  $|I|\leq n-1$.

We consider the graphs $G$ satisfying (\romannumeral2). If $I\subseteq\{u_1,u_2,\dots,u_n,v_n\}$, then $|I|\leq n-1$ for $G[\{u_1,u_2,\dots,u_n,v_n\}]$ has at
least two independent edges. Otherwise, there exists $v_k\in I$ for some $1\leq k\leq n-1$. Since $v_kv_l\notin E(G)$ for any $1\leq l\leq n-1$ and $l\neq k$, $v_ku_l\in E(G)$ since $G[\{u_k,v_k,u_l,v_l\}]$ contains an $M$-alternating 4-cycle by Lemma \ref{2.2}. So $u_l\notin I$ and $I\subseteq \{v_1,v_2,\dots,v_n,u_n\}$. By the assumption, $\{v_ju_n, v_{i}v_n\}\subseteq E(G)$.
If $i\neq j$, then $\{v_ju_n, v_{i}v_n\}$ are two independent edges and $|I|\leq n-1$. Otherwise, $G[\{v_{i},v_n,u_n\}]$ is a triangle and $|I|\leq n-1$.

Necessity. By Lemma \ref{l4}, there exists a subset $S\subseteq V(G)$ such that $G[S]$ contains two independent edges, all components of $G-S$ are factor-critical,
and
$o(G-S)=|S|-2$. Hence $|S|\geq 4$. Let $C_i~(1\leq i\leq |S|-2)$ be all components of $G-S$. Since $F(G)=n-1$, there exists a perfect matching $M$ of $G$ so that $f(G,M)=n-1$. We claim that
$G-S$ has at most one component containing exactly three vertices, and the others are trivial. Otherwise, $G-S$ has one component, say $C_1$, containing at least five vertices or $G-S$ has two  components, say $C_1$ and $C_2$, containing exactly three vertices.
Since $G-S$ has exactly $|S|-2$ factor-critical components, in either case there is one
edge $e_1$ in $M$ belonging to $C_1$ or $C_2$. Then none of end vertices of $e_1$ are adjacent to any vertex, say $w$, of another component. So $G[V(\{e_1,e_2\})]$ cannot contain an $M$-alternating cycle, where $e_2$ is the edge in $M$ incident with $w$, which contradicting Lemma \ref{2.2}.

If one component of $G-S$ is a triangle and the other components are singletons, then $|S|=n\geq4$. Let $V(C_i)=\{v_i\}$ for each $1\leq i\leq n-3$,
$V(C_{n-2})=\{v_{n-2},v_{n-1},v_n\}$ and $u_iv_i\in M$ for each $1\leq i\leq n-3$. Since $f(G,M)=n-1$, each edge of $C_{n-2}$ does not belong to $M$ by Lemma \ref{2.2}.
So we denote the remaining three edges of $M$ by $u_{n-2}v_{n-2},u_{n-1}v_{n-1}$ and $u_nv_n$. Thus $M=\{u_iv_i|i=1,2,\dots,n\}$ and $S=\{u_1,u_2,\dots,u_n\}$. Hence (\romannumeral1) holds.

If all components of $G-S$ are trivial, then $|S|=n+1$ and $o(G-S)=n-1$. Let $V(C_i)=\{v_i\}$ and $u_iv_i\in M$ for each $1\leq i\leq n-1$. Then the remaining edge of $M$ is
denoted by $u_nv_n$. So $\{v_1,v_2,\dots,v_{n-1}\}$ is an independent set of $G$ and $S=\{u_1,u_2,\dots,u_n,v_n\}$ with $G[S]$ containing at least two independent edges. Since $G$ is not a graph in $\mathcal{K}_{n,n}^+$, $G$ contains no independent set of size $n$ by Lemma \ref{p2}. Combining that $\{v_1,v_2,\dots,v_{n-1}\}$ is an independent set of $G$, there exists $i$ and $j$ with $1\leq i,j\leq n-1$ such that $\{u_nv_j,v_nv_i\}$
is contained in $E(G)$. So (\romannumeral2) holds.
\end{proof}

Next we will determine all minimal non-2-extendable graphs in Theorem \ref{p3}.
\begin{cor}
Let graph $G$ in Theorem \ref{p3} be minimal. Then $G$ is not 2-extendable if and only if (\romannumeral2) in Theorem \ref{p3} holds where $i\neq j$,
$\{u_1,u_2,\dots,u_{n-1}\}$ is an independent set and $G[\{u_n,v_n,u_i,v_i\}]$ is just a 4-cycle for $1\leq i\leq n-1$ (see an example in Fig. \ref{p91111}(c)).
\end{cor}
\begin{proof}Sufficiency. It suffices to prove that $G$ is minimal by Theorem \ref{p3}. Since $f(G,M)=n-1$,  by Lemma \ref{2.2}, $G[\{u_i,v_i,u_j,v_j\}]$ contains an $M$-alternating 4-cycle for any $1\leq i< j\leq n-1$. By the assumption, $\{u_1,u_2,\dots,u_{n-1}\}$ and $\{v_1,v_2,\dots,v_{n-1}\}$ are two independent sets. So neither $v_iv_j$ nor $u_iu_j$ is an edge of $G$. Hence $G[\{u_i,v_i,u_j,v_j\}]$ is exactly a 4-cycle. Combining that $G[\{u_n,v_n,u_i,v_i\}]$ is just a 4-cycle for $1\leq i\leq n-1$, we obtain that $G$ is minimal by Lemma \ref{2.2}.

Necessity. By Theorem \ref{p3}, $G$ has a perfect matching $M=\{u_iv_i| i=1,2,\dots,n\}$ with $f(G,M)=n-1$ so that (\romannumeral1) or (\romannumeral2)
holds. First we show that each graph $G$ satisfying (\romannumeral1) is not minimal. Let $\{v_1,v_2,\dots,v_{n-3}\}$ be the set of $n-3$ isolated vertices and $G[\{v_{n-2},v_{n-1},v_n\}]$ be the triangle of $G[\{v_{1},v_{2},\dots,v_n\}]$.
Then there is at least one vertex, say $u_i$ for some $1\leq i \leq n-3$, incident with one of the two independent edges of
$G[\{u_1,u_2,\dots,u_n\}]$. Assume that $u_iu_j$ is such an edge for some $1\leq j\leq n$. By Lemma \ref{2.2}, $G[\{u_i,v_i,u_j,v_j\}]$ contains an $M$-alternating 4-cycle. Since $v_iv_j\notin E(G)$, we have $\{u_iv_j,v_iu_j\}\subseteq E(G)$. But $u_iu_j\in E(G)$, $G[\{u_i,v_i,u_j,v_j\}]$ is not a 4-cycle.
So $G$ is not minimal by Lemma \ref{2.2}.

Next let $G$ be a graph satisfying (\romannumeral2). Since $G$ is minimal, $G[\{u_k,v_k,u_l,v_l\}]$ is just a 4-cycle for $1\leq k<l\leq n$ by Lemma \ref{2.2}.
Since
$\{v_iv_n,v_ju_n\}\subseteq E(G)$ for some $i$ and $j$ with $1\leq i,j\leq n-1$, we have $i\neq j$. Since $\{v_{1},v_{2},\dots,v_{n-1}\}$ is an independent
set, $v_kv_l\notin E(G)$ for $1\leq k< l \leq n-1$. By Lemma \ref{2.2}, $G[\{u_k,v_k,u_l,v_l\}]$ is a 4-cycle $u_kv_lu_lv_ku_k$ and $u_ku_l\notin E(G)$. Hence $\{u_1,u_2,\dots,u_{n-1}\}$ is an independent set.
\end{proof}

\section{\normalsize Minimum forcing numbers and forcing spectrum of graphs $G$ with $F(G)=n-1$}
The \emph{forcing spectrum} of a graph $G$ is the set of forcing numbers of all perfect matchings of graph $G$. If the forcing spectrum is an integer interval, then we say it is  \emph{continuous} (or \emph{ consecutive}).  Afshani et al. \cite{5} showed that any finite
subset of positive integers is the forcing spectrum of some graph. Besides, they \cite{5} obtained that the forcing spectra of column continuous subgrids are  continuous by matching 2-switches. Further, Zhang and Jiang \cite{41} generalized their result to any polyomino with perfect matchings by applying the $Z$-transformation graph. Zhang and Deng \cite{30} obtained that the forcing spectrum of any hexagonal system with a forcing edge form either the
integer interval from 1 to its Clar number or with only the gap 2.
For more researches on the forcing spectra of special graphs, see \cite{4,5,19,31,32,33,ZZ22}.

Let $G\in\mathcal{G}_{2n}$ with $F(G)=n-1$. In this section we will prove that $f(G)\geq \lfloor\frac{n}{2}\rfloor$, and find that for the class of such graphs all
minimum forcing numbers of them form an integer interval $[\lfloor\frac{n}{2}\rfloor,n-1]$. Further, we will show that the forcing spectrum of each such graph $G$ is
continuous.
Next we give a lemma obtained by Che and Chen.
\begin{lem}[\cite{10}]\label{4.1} Let $G$ be a $k$-connected graph with a perfect matching. Then $f(G)\geq \lfloor\frac{k}{2}\rfloor$.
\end{lem}
Combining Lemmas \ref{p1} and \ref{4.1}, we have the following result.
\begin{cor}\label{c3}Let $G\in\mathcal{G}_{2n}$ with $F(G)=n-1$. Then
$f(G)\geq \lfloor\frac{n}{2}\rfloor$.
\end{cor}
For $0\leq k\leq\lfloor\frac{n-1}{2}\rfloor$, let $H_k\in \mathcal{G}_{2n}$ be a minimal graph with $f(H_k,M_0)=n-1$, where $M_0=\{u_iv_i|i=1,2,\dots,n\}$ is a perfect
matching of $H_k$, so that $H_k[\{u_1,u_2,\dots,u_n\}]$ contains exactly $k$ edges $\{u_{2i-1}u_{2i} |i=1,2,\dots,k\}$. Then $H_k[\{v_1,v_2,\dots,v_n\}]$
contains exactly $k$ edges $\{v_{2i-1}v_{2i}| i=1,2,\dots,k\}$ and $H_k[\{u_{2k+1},v_{2k+1},\dots,u_n,v_n\}]$ is isomorphic to a complete bipartite graph (see
examples in Fig. \ref{p442}).
 \begin{figure}
\centering
\includegraphics[height=3cm,width=14.3cm]{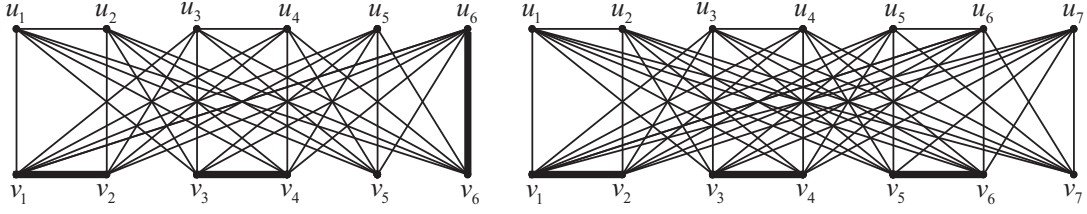}
\caption{\label{p442}$H_k$ with $k=\lfloor\frac{n-1}{2}\rfloor$ and $n=6$ and 7.}
\end{figure}
\begin{rem}\label{r5} \textrm{For} $0\leq k\leq\lfloor\frac{n-1}{2}\rfloor$, $f(H_k)\leq n-k-1.$ \textrm{Especially, for} $k=\lfloor\frac{n-1}{2}\rfloor$ \textrm{we have}
$f(H_k)=\lfloor\frac{n}{2}\rfloor$, \textrm{i.e., the lower bound of Corollary \ref{c3} is sharp.}

 \textrm{Let} $M=\{u_{2i-1}u_{2i},v_{2i-1}v_{2i} |i=1,2,\dots,k\}\cup\{u_{2k+1}v_{2k+1},u_{2k+2}v_{2k+2},\dots,u_nv_n\}$ \textrm{be a perfect matching of $H_k$. It follows that}
 $S=\{v_{2i-1}v_{2i}|i=1,2,\dots,k\}\cup\{u_{2k+2}v_{2k+2},u_{2k+3}v_{2k+3},\\ \dots,u_nv_n\}$ \textrm{is a forcing set of $M$ by Lemma \ref{2.1} (see examples in Fig.}
 \textrm{\ref{p442}, where bold lines form $S$). So} $f(H_k)\leq f(H_k,M)\leq |S|=k+(n-2k-1)=n-k-1$.

 \textrm{Especially, for} $k=\lfloor\frac{n-1}{2}\rfloor$ \textrm{we have} $f(H_k)\leq n-k-1=\lfloor\frac{n}{2}\rfloor$. \textrm{Combining Corollary \ref{c3}, we obtain that}
 $f(H_k)=\lfloor\frac{n}{2}\rfloor$.
\end{rem}

Next we will prove that $f(H_k)=n-k-1$ for $0\leq k\leq\lfloor\frac{n-1}{2}\rfloor$.
Here we give two simple facts that can be obtained from Lemma \ref{2.1}.
\begin{fact}\label{Fact 1}If $G'$ is a spanning subgraph of $G$ with $f(G',M)=|S|$ for some perfect matching $M$ of $G'$, then $f(G,M)\geq |S|$.
\end{fact}
\begin{fact}\label{Fact 2}Let $M$ be a perfect matching of $G$ with $M=M_1\cup M_2$, and let $G_i=G[V(M_i)]$ for $i=1,2$. Then $f(G,M)\geq f(G_1,M_1)+f(G_2,M_2)$.
\end{fact}
\begin{lem}\label{l5} For $0\leq k\leq \lfloor\frac{n-1}{2}\rfloor$, $f(H_k)= n-k-1$.
\end{lem}
\begin{proof}By Remark \ref{r5},
it suffices to prove that $f(H_k)\geq n-k-1$. We proceed by induction on $n$. It is trivial for $n=1$. Since $H_0\cong K_{n,n}$, we have $f(H_0)=n-1$ by Theorem \ref{thm1.3}. By Remark
\ref{r5}, for  $k=\lfloor\frac{n-1}{2}\rfloor$ we have $f(H_k)= \lfloor\frac{n}{2}\rfloor =n-k-1$. So next we suppose  $n\geq 2$ and $1\leq k\leq \frac{n-3}{2}$. Let
$M$ be any perfect matching of $H_k$. Then we have the following claims.

{\textbf{Claim 1.} If $u_iv_i\in M$ for some $1\leq i\leq 2k$, then $f(H_k,M)\geq n-k-1$.}

Let $H^2=H_k-\{u_i,v_i\}$ and $M_2=M\setminus \{u_iv_i\}$. Then $H^2$ is isomorphic to $H'_{k'}$ with $n'=n-1$ and $k'=k-1$. By the induction hypothesis,
$$f(H^2,M_2)=f(H'_{k'},M_2)\geq f(H'_{k'})\geq n'-k'-1=n-k-1.$$ By Fact \ref{Fact 2}, $f(H_k,M)\geq f(H^2,M_2)\geq  n-k-1$.

{\textbf{Claim 2.} If $G$ has an $M$-alternating 4-cycle containing exactly two edges of $\{u_{2i-1}u_{2i},\\v_{2i-1}v_{2i} |i=1,2,\dots,k\}$, then $f(H_k,M)\geq
n-k-1$.}

If $M\cap \{u_{2i-1}u_{2i},v_{2i-1}v_{2i} |i=1,2,\dots,k\}=\emptyset$, then we assume that $C$ is an $M$-alternating 4-cycle of $G$, and $\{u_{2i-1}u_{2i},v_{2j-1}v_{2j}\}$ is contained in $C$ for some $1\leq i,j\leq k$. So $C=u_{2i-1}v_{2j-1}v_{2j}u_{2i}u_{2i-1}$ or $u_{2i-1}v_{2j}v_{2j-1}u_{2i}u_{2i-1}$. Let $H^2=H_k-V(C)$
and $M_2=M\cap E(H^2)$, $M_1=M\setminus M_2$, $H^1=H_k[V(M_1)]$. Then $H^1$ contains $C$ and $f(H^1,M_1)\geq 1$ by Theorem \ref{thm6}. On the other hand,
$H^2-\{v_{2i-1}v_{2i},u_{2j-1}u_{2j}\}$ is isomorphic to $H'_{k'}$ with $n'=n-2$, $k'\leq k-1$, and $M_2$ is a perfect matching of $H'_{k'}$. By the induction
hypothesis and Fact \ref{Fact 1}, $$f(H^2,M_2)\geq f(H'_{k'},M_2)\geq f(H_{k'})\geq n'-k'-1\geq n-k-2.$$ By Fact \ref{Fact 2}, $f(H_k,M)\geq f(H^1,M_1)+f(H^2,M_2)\geq
1+(n-k-2)= n-k-1.$

If $M\cap \{u_{2i-1}u_{2i},v_{2i-1}v_{2i} |i=1,2,\dots,k\}\neq \emptyset$, then the intersection of the two sets is denoted by $M_1$. By the structure of $H_k$,
all vertices of $\{u_1,u_2,\dots,u_n\}\setminus V(M_1)$ must match into all vertices of $\{v_1,v_2,\dots,v_n\}\setminus V(M_1)$. So
we have  $$|M_1\cap\{u_{2j-1}u_{2j}|j=1,2,\dots,k\}|=|M_1\cap\{v_{2j-1}v_{2j}|j=1,2,\dots,k\}|.$$ Let $M_2=M\setminus M_1$ and $H^i=H_k[V(M_i)]$ for $i=1$, 2.
For any pair of edges $\{u_{2i-1}u_{2i},v_{2j-1}v_{2j}\}$ of $M_1$ where $1\leq i,j\leq k$, $G[\{u_{2i-1},u_{2i},v_{2j-1},v_{2j}\}]$ is either a 4-cycle or $K_4$ by the structure of $H_k$. So $G[\{u_{2i-1},u_{2i},v_{2j-1},v_{2j}\}]$ contains an $M_1$-alternating 4-cycle, and so $H^1$ contains $\frac{|M_1|}{2}$ disjoint $M_1$-alternating 4-cycles. By Lemma  \ref{2.1}, $f(H^1,M_1)\geq\frac{|M_1|}{2}$. On the other hand,
$$H^2-\{u_{2j-1}u_{2j}|v_{2j-1}v_{2j}\in M_1, j =1,2,\dots,k\}\cup \{v_{2j-1}v_{2j}|u_{2j-1}u_{2j}\in M_1, j=1,2,\dots,k\}$$ is isomorphic to $H'_{k'}$ with
$n'=n-|M_1|,~k'\leq k-\frac{|M_1|}{2}$. Since $$k'\leq k-\frac{|M_1|}{2}\leq \frac{n-3-|M_1|}{2}<
\frac{n-2-|M_1|}{2}\leq\lfloor\frac{n-1-|M_1|}{2}\rfloor=\lfloor\frac{n'-1}{2}\rfloor,$$ by the induction hypothesis and Fact \ref{Fact 1}, we obtain that
$$f(H^2,M_2)\geq f(H'_{k'},M_2)\geq f(H'_{k'})\geq n'-k'-1\geq n-k-1-\frac{|M_1|}{2}.$$ By Fact \ref{Fact 2}, we obtain that $$f(H_{k},M)\geq
f(H^1,M_1)+f(H^2,M_2)\geq \frac{|M_1|}{2}+(n-k-1-\frac{|M_1|}{2})=n-k-1.$$

By Claims 1 and 2, from now on we suppose that $M$ contains no $u_iv_i$ for any $1\leq i\leq2k$ and any $M$-alternating 4-cycle of $G$ contains at most one edge of
$\{u_{2i-1}u_{2i},v_{2i-1}v_{2i} |i=1,2,\dots,k\}$. Particularly, $M$ contains no edges of $\{u_{2i-1}u_{2i},v_{2i-1}v_{2i} |i=1,2,\dots,k\}$. So we may assume
$\{v_1u_3,v_2u_5\}$ is contained in $M$.
Next we are going to consider the edge of $M$ incident with $v_{3}$, say $v_3u_l$. Then $l$ has four possible values: $l\leq 2$, $l=6$, $7\leq l\leq 2k$ and $l\geq
2k+1$. Furthermore, if $7\leq l\leq 2k$, then we suppose $l=7$ and continue to consider the edge of $M$ incident with $v_{5}$. Until we obtain an edge $v_hu_l \in M$
so that one of the other three cases ($l\leq h-1$, $l=h+3$, $l\geq 2k+1$) holds (since $H_k$ is finite, such edge exists).
\begin{figure}
\centering
\includegraphics[height=6cm,width=14cm]{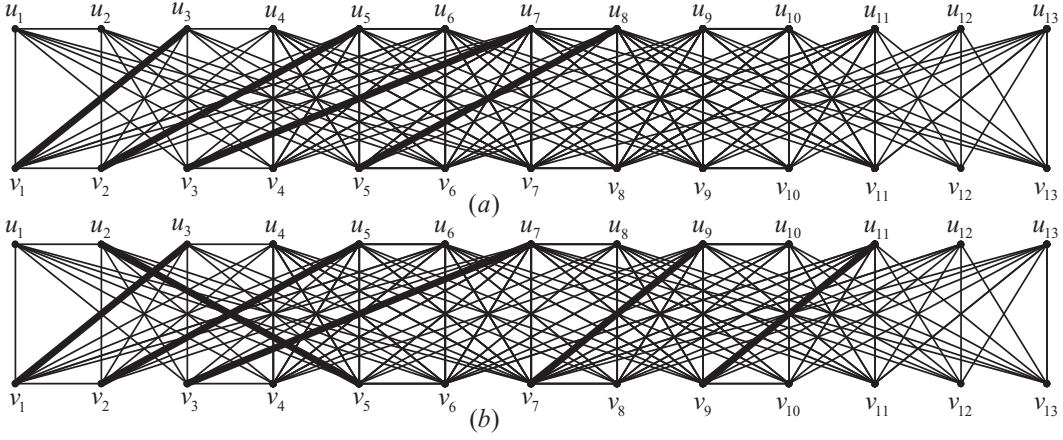}
\caption{\label{p511}Illustration for the proof of Lemma \ref{l5}, where bold lines form $M_1$.}
\end{figure}

If $l=h+3$, then let $M_1=\{v_1u_3,v_2u_5,\dots, v_hu_{h+3}\}$ (see an example in Fig. \ref{p511}(a) where $h=5$), $M_2=M\setminus M_1$, and $H^i=H_k[V(M_i)]$ for
$i=1, 2$. Then $H^1$ contains an $M_1$-alternating 4-cycle $v_1u_5v_2u_3v_1$. By Theorem \ref{thm6}, $f(H^1,M_1)\geq 1$. On the other hand,
$H^2-\{u_{1}u_{2},v_{h+2}v_{h+3}\}$ is isomorphic to $H'_{k'}$ with $n'=n-|M_1|$. Since contribution of the edges $v_1u_3$, $v_hu_{h+3}$ to $k'$ is -2, 0, and that of
each other edges in $M_1$ is -1, we have $k'=k-|M_1|$.
Since $$k'= k-|M_1|\leq \frac{n-3-2|M_1|}{2}< \frac{n-2-|M_1|}{2}\leq\lfloor\frac{n-1-|M_1|}{2}\rfloor=\lfloor\frac{n'-1}{2}\rfloor,$$ by the induction hypothesis
and Fact \ref{Fact 1}, $$f(H^2,M_2)\geq f(H'_{k'},M_2)\geq f(H'_{k'})\geq n'-k'-1=n-k-1.$$ By Fact \ref{Fact 2}, $f(H_k,M)\geq f(H^1,M_1)+f(H^2,M_2)\geq 1+n-k-1=
n-k$.

If $l\in [2k+1,n]\cup[1,h-1]$ (see an example in Fig. \ref{p511}(b) where $h=5,~l=2)$, then we continue to consider the edge of $M$ incident with $v_{h+2}$,
say $v_{h+2}u_i$. Then $i$ has three possible values: $i\leq h+1$, $h+4\leq i\leq 2k$ and $i\geq 2k+1$. Furthermore, if $h+4\leq i\leq 2k$, then we suppose $i=h+4$
and continue to consider the edge of $M$ incident with $v_{h+4}$. Until we obtain an edge of $M$, say $v_{r}u_t$ with $t\in [2k+1,n]\cup [1,r-1]$ (By the finiteness
of $G$, such edge exists) (see an example in Fig. \ref{p511}(b) where $r=9,~t=11)$. Let $M_1=\{v_1u_3,v_2u_5,\dots,v_hu_l,\dots,v_{r}u_t\}$, $M_2=M\setminus M_1$, and
$H^i=H_k[V(M_i)]$ for $i=1, 2$. Then
$H^2-u_{1}u_{2}$ is isomorphic to $H'_{k'}$ with $n'=n-|M_1|$. Since contribution of the edges $v_1u_3$, $v_hu_l$, $v_ru_t$ to $k'$ is -2, 0, 0 and that of each other
edges in $M_1$ is -1, we have $k'=k-(|M_1|-1)$. Since $$k'= k-|M_1|+1\leq \frac{n-1-2|M_1|}{2}\leq
\frac{n-2-|M_1|}{2}\leq\lfloor\frac{n-1-|M_1|}{2}\rfloor=\lfloor\frac{n'-1}{2}\rfloor,$$ by the induction hypothesis and Fact \ref{Fact 1}, $$f(H^2,M_2)\geq
f(H'_{k'},M_2)\geq f(H'_{k'})\geq n'-k'-1=n-k-2.$$ By Fact \ref{Fact 2}, $f(H_k,M)\geq f(H^1,M_1)+f(H^2,M_2)\geq 1+(n-k-2)= n-k-1$.

By the arbitrariness of $M$, we obtain that $f(H_k)\geq n-k-1$.
 \end{proof}

Combining Lemma \ref{l5} and Corollary \ref{c3}, we obtain the following result.
\begin{thm} All minimum forcing numbers of graphs $G\in \mathcal{G}_{2n}$ with $F(G)=n-1$ form an integer interval $[\lfloor\frac{n}{2}\rfloor,n-1]$.
\end{thm}

Suppose that $M$ is a perfect matching of $G$. If $C$ is an $M$-alternating cycle of length 4, then $M\oplus E(C)$ is a \emph{matching 2-switch} on $M$.
Afshani et al. \cite{5} obtained that a matching 2-switch on a perfect matching does not change the forcing number by more than 1.
\begin{lem}[\cite{5}]\label{4.3} If $M$ is a perfect matching of $G$ and $C$ is an $M$-alternating cycle of length 4, then $$|f(G,M\oplus E(C))-f(G,M)|\leq 1.$$
\end{lem}
By Lemma \ref{4.3}, if $M_1,M_2,\dots,M_s$ is a sequence of perfect matchings such that $M_{i+1}$ is obtained from $M_i$ by a matching 2-switch for $1\leq i\leq s-1$,
then the integer interval [min$\{f(G,M_1), f(G,M_s)\}$,
max$\{f(G,M_1),f(G,M_s)\}]$ is contained in the forcing spectrum of $G$.
\begin{thm}\label{thm2}If $G\in \mathcal{G}_{2n}$ with $F(G)=n-1$, then its forcing spectrum is continuous.
\end{thm}
\begin{proof}Let $M_s=\{u_iv_i|i=1,2,\dots,n\}$ be a perfect matching of $G$ with $f(G,M_s)=n-1$. Then we will prove that $M_s$ can be obtained from any perfect
matching $M$ of $G$ by repeatedly applying matching 2-switches. If we have done, then $M_s$ can be obtained from a perfect matching of $G$ with the minimum forcing
number by repeatedly applying matching 2-switches. So the forcing spectrum of $G$ is continuous by Lemma \ref{4.3}.

We proceed by induction on $n$. For $n=1$, $G\cong K_2$ and the result is trivial. Next, for $n\geq 2$, $f(G, M_s)=n-1>0$. Take any perfect matching $M$  of $G$ different from $M_s$. Then we have the following claims.

{\textbf{Claim 1.} If $M_s\cap M\neq\emptyset$, then $M_s$ can be obtained from $M$ by repeatedly applying matching 2-switches.}

Suppose $M_s\cap M$ contains an edge $u_iv_i$ for some $1\leq i\leq n$.
Let $G'=G-\{u_i,v_i\}$ and $M'=M\setminus\{u_iv_i\}$, $M_s'=M_s\setminus\{u_iv_i\}$. Then $G'$ has $2(n-1)$ vertices, $M'$ and $M_s'$ are two distinct perfect matchings of $G'$. By Lemma
\ref{2.2}, $f(G',M_s')=n-2$.
By the induction hypothesis, $M_s'$ can be obtained from $M'$ by repeatedly applying matching 2-switches. Hence, $M_s$ is also obtained from $M$ by the same series of
matching 2-switches, and the claim holds.

{\textbf{Claim 2.} If $M_s\oplus M$ contains a cycle of length 4, then $M_s$ can be obtained from $M$ by repeatedly applying matching 2-switches.}

Assume that $C$ is a cycle of length 4 contained in $M_s\oplus M$ and $V(C)=\{u_i,v_i,u_j,v_j\}$. Let $M'=M\oplus E(C)$. Then $\{u_iv_i,u_jv_j\}\subseteq M_s\cap M'$. By
Claim  1, $M_s$ can be obtained from $M'$ by repeatedly applying matching 2-switches.
Hence $M_s$ can be obtained from $M$ by repeatedly applying matching 2-switches, and the claim holds.

By Claims 1 and 2, from now on we suppose that $M_s\cap M=\emptyset$ and $M_s\oplus M$ contains no cycles of length 4. Then we may suppose that $v_1u_2\in M$.
So $u_1v_2\notin M$. Without loss of generality, we assume that $u_1v_3\in M$. If $u_2v_3$ is an edge of $G$, then $C_1=u_1v_1u_2v_3u_1$ is an $M$-alternating 4-cycle. So  $M'=M\oplus E(C_1)$ is a perfect matching of $G$ and $M_s\cap M'=\{u_1v_1\}$. From Claim 1, we are done. Otherwise, $u_2v_3$ is not an edge of $G$. By Lemma \ref{2.2}, $\{u_2u_3,v_2v_3\}\subseteq E(G)$ (see Fig. \ref{p91}(a)). Next we consider the following two cases according as whether $v_2u_3$ belongs to $M$ or not.

\textbf{Case 1.} $v_2u_3\in M$. If $v_1u_3\in E(G)$, then $C_2=v_1u_2v_2u_3v_1$ is an $M$-alternating 4-cycle. So $M'=M\oplus E(C_2)$ is a perfect matching of $G$ and $M_s\cap M'=\{u_2v_2\}$. From Claim 1, we are done. Otherwise, $v_1u_3\notin E(G)$. By Lemma \ref{2.2}, $\{u_1u_3,v_1v_3\}\subseteq E(G)$ (see Fig. \ref{p91}(b)).
Then $C_3=u_1u_3v_2v_3u_1$ is an $M$-alternating 4-cycle and $C_4=v_2v_3v_1u_2v_2$ is an $M\oplus E(C_3)$-alternating 4-cycle. So $M'=M\oplus E(C_3)\oplus E(C_4)$ is a perfect matching of $G$ that is obtained from $M$ by two matching 2-switches and $M_s\cap M'=\{u_2v_2\}$. From Claim 1, we are done.

\textbf{Case 2.} $v_2u_3\notin M$. Without loss of generality, we can suppose that $v_2u_4\in M$. If $v_1u_4$ is an edge of $G$, then $C_5=v_1u_2v_2u_4v_1$ is an $M$-alternating 4-cycle. So
$M'=M\oplus E(C_5)$ is a perfect matching of $G$ and $M_s\cap M'=\{u_2v_2\}$. From Claim 1, we are done. Otherwise, $v_1u_4$ is not an edge of $G$. By Lemma \ref{2.2}, $\{u_1u_4,v_1v_4\}\subseteq E(G)$ (see Fig. \ref{p91}(c)). Next we distinguish the following two subcases according to $u_3v_4 \in M$ or not.
\begin{figure}
\centering
\includegraphics[height=6.5cm,width=14cm]{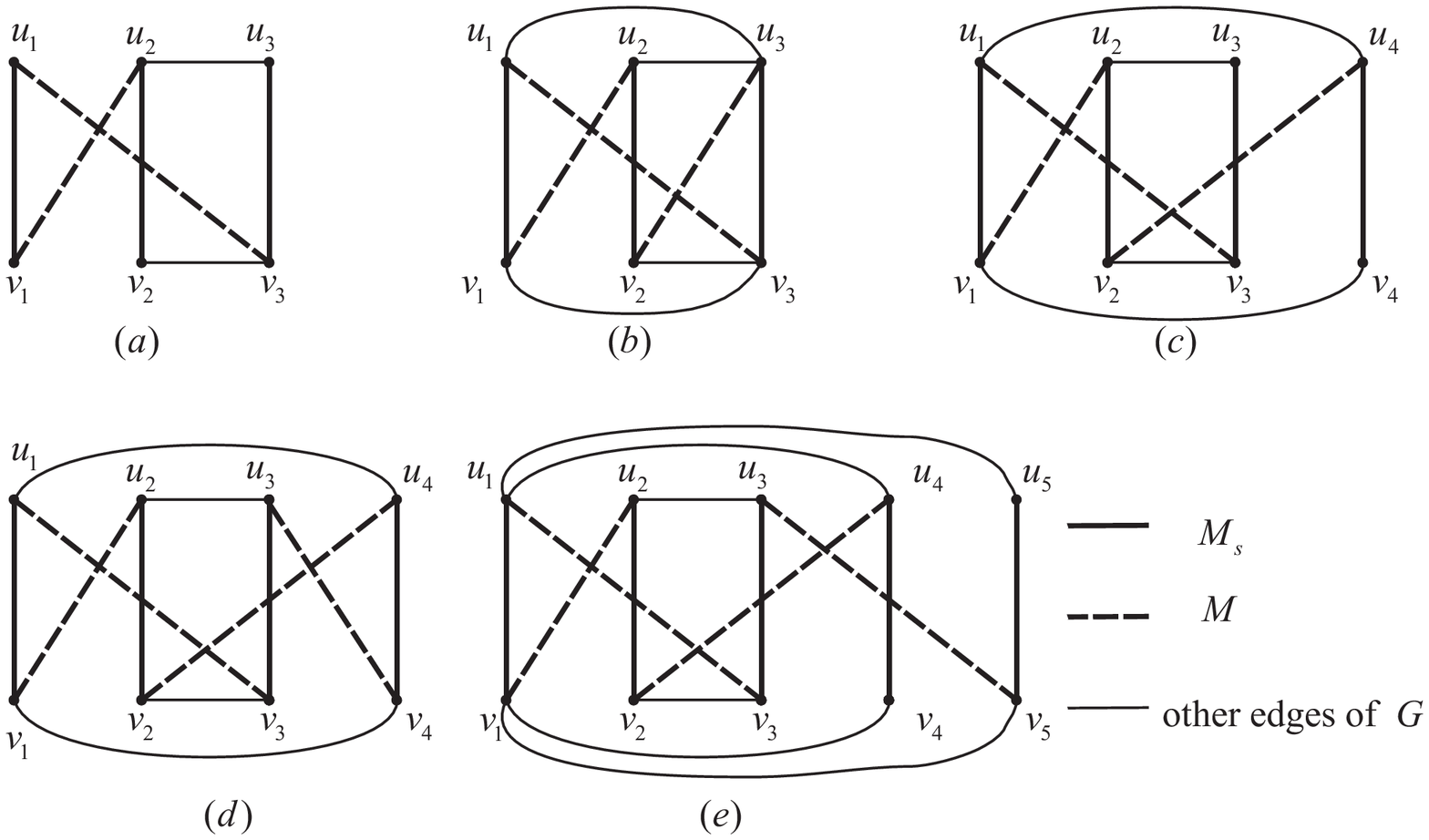}
\caption{\label{p91}Illustration for the proof of Theorem \ref{thm2}.}
\end{figure}

\textbf{Subcase 2.1.} $u_3v_4\in M$ (see Fig. \ref{p91}(d)). Then $C_6=u_1v_3v_2u_4u_1$ and $C_7=v_1v_4u_3u_2v_1$ are two $M$-alternating 4-cycles. By two matching 2-switches we have that $M'=M\oplus E(C_6) \oplus E(C_7)$ is a perfect matching of $G$. Hence, $M'=(M\setminus \{u_1v_3,v_2u_4,v_1u_2,u_3v_4\})\cup \{v_2v_3,u_1u_4,u_2u_3,v_1v_4\}$. So $M_s\oplus M'$ contains a 4-cycle $u_2u_3v_3v_2u_2$. From Claim 2, $M_s$ can be obtained from $M'$ by repeatedly applying matching 2-switches, so we are done.

\textbf{Subcase 2.2.} $u_3v_4\notin M$. Then we may suppose $u_3v_5\in M$. If $u_1v_5$ is an edge of $G$, then $C_8=u_1v_5u_3v_3u_1$ is an $M$-alternating 4-cycle. So $M'=M\oplus E(C_8)$ is a perfect matching of $G$ and $M_s\cap M'=\{u_3v_3\}$. From Claim 1, we are done. So we may suppose $u_1v_5\notin E(G)$. By Lemma \ref{2.2}, $\{v_1v_5,u_1u_5\}\subseteq E(G)$ (see Fig. \ref{p91}(e)). Then $C_9=u_2u_3v_5v_1u_2$ and $C_{6}$ are two $M$-alternating 4-cycles. By two matching 2-switches we have that $M'=M\oplus E(C_9) \oplus E(C_{6})$ is a perfect matching of $G$. Hence, $M'=(M\setminus \{u_1v_3,v_2u_4,v_1u_2,u_3v_5\})\cup \{v_2v_3,u_1u_4,u_2u_3,v_1v_5\}$. So $M_s\oplus M'$ contains a 4-cycle $u_2u_3v_3v_2u_2$. From Claim 2, we are done.
\end{proof}

\noindent{\normalsize \textbf{Acknowledgments}}

The authors are grateful to  anonymous reviewers for giving valuable comments and suggestions in improving the manuscript.

\end{document}